\documentclass[12pt,reqno]{amsart}
\usepackage{amsthm,amsmath,amsfonts, amssymb,url,cite,color}
\usepackage{mathrsfs}
\usepackage{enumerate}
\numberwithin{equation}{section}

\usepackage{tikz}
\usepackage{tikzscale}
\usepackage{graphicx}
\usepackage{graphics}
\usepackage{pict2e}
\usepackage{collectbox}
\usepackage{cite}

\usepackage[mathscr]{euscript}

\usepackage{bm}

\newtheorem{thm}{Theorem}
\newtheorem{lem}[thm]{Lemma}
\newtheorem{prop}[thm]{Proposition}

\newtheorem{defi}[thm]{Definition}

\newtheorem{?}[thm]{Problem}

\theoremstyle{remark}
\newtheorem{remark}[thm]{Remark}
\newtheorem{example}[thm]{Example}

\def\L{\mathbf{L}}
\def\la{\lambda}

\def\cyc{\mathsf{cyc}}

\def\wex{\mathsf{wex}}

\def\inv{\mathsf{inv}}
\def\ind{\mathsf{ind}}

\def\rl{\mathsf{rl}}

\def\rec{\mathsf{rec}}
\def\cros{\mathsf{cros}}
\def\BC{\mathcal{BC}}
\def\cw{\mathsf{cw}}

\def\cs{\mathsf{cd}}

\def\C{\mathbb{C}}
\def\Z{\mathbb{Z}}
\def\N{\mathbb{N}}

\def\F{\mathcal{F}}

\def\Lin{\mathcal Lin}
\def\S{{\mathfrak S}}

\newcommand{\bea}{\begin{align}} 
\newcommand{\ena}{\end{align}}

\newcommand\lc{\mathcal{LC}}
\begin{document}

\title[Combinatorics of $(q,y)$-Laguerre  polynomials]{Combinatorics
 of  $(q,y)$-Laguerre polynomials and their moments}
\author{Qiongqiong Pan}
\address{Univ Lyon,  Universit\'e Lyon1,
UMR 5208 du CNRS, Institut Camille Jordan\\
F-69622, Villeurbanne Cedex, France}
\email{qpan@math.univ-lyon1.fr}
\author{Jiang Zeng}\address{Univ Lyon,  Universit\'e Lyon1,
UMR 5208 du CNRS, Institut Camille Jordan\\
F-69622, Villeurbanne Cedex, France}
\email{zeng@math.univ-lyon1.fr}
\dedicatory{To Christian  Krattenthaler  on the occasion of his 60th birthday}

\begin{abstract} 
We consider a $(q,y)$-analogue of Laguerre polynomials 
$L^{(\alpha)}_n(x;y\,|\,q)$  for  integral  $\alpha\geq -1$, which turns out to be  a rescaled version of Al-Salam--Chihara polynomials. 
A combinatorial interpretation  for  the $(q,y)$-Laguerre polynomials  is given  using a colored version of Foata and Strehl's Laguerre configurations with suitable statistics.
When $\alpha\geq 0$,  the  corresponding moments  are described using certain 
classical statistics on permutations, and 
 the linearization coefficients   are  proved to be a polynomial in $y$ and $q$ with nonnegative integral coefficients.
\end{abstract}
\maketitle



\thispagestyle{myheadings}
\font\rms=cmr8 
\font\its=cmti8 
\font\bfs=cmbx8

\markright{\its S\'eminaire Lotharingien de
Combinatoire \bfs 81 \rms (2020), Article~B81e\hfill}
\def\thepage{}


\section{Introduction}
The  monic Laguerre polynomials $L^{(\alpha)}_n(x)$  are defined 
by the generating function
\begin{align}\label{gf:laguerre}
(1+t)^{-\alpha-1}\exp\left(\frac{xt}{t+1}\right)=\sum_{n=0}^\infty L^{(\alpha)}_n(x) \frac{t^n}{n!}.
\end{align}
They are   the multiple of the usual (general) Laguerre polynomials~\cite[pp.~241--242]{KLK10}    by
 $(-1)^n n!$. We have  the   explicit formula 
\begin{equation}\label{eq:explag}
L^{(\alpha)}_n(x)=\sum_{k=0}^{n} (-1)^{n-k}\frac{n!}{k!}\,\binom{n+\alpha}{  n-k} x^{k}
\end{equation}
and the   three-term recurrence relation
\begin{align}\label{eq:recurr}
L^{(\alpha)}_{n+1}(x)=(x-(2n+\alpha+1))L^{(\alpha)}_n(x)-n(n+\alpha)L^{(\alpha)}_{n-1}(x).
\end{align}
The Laguerre polynomials $L_n^{(\alpha)}(x)$ are  orthogonal with respect to 
the moments 
${\mathcal L}(x^n)=(\alpha+1)_n$, where $(x)_n=x(x+1)\cdots (x+n-1)$ ($n\geq 1$) is the shifted  factorial with 
$(x)_0=1$,  and  ${\mathcal L}$ is the linear functional defined by
\begin{equation}\label{eq:momentsimp}
{\mathcal L}(f)=\frac{1}{\Gamma(\alpha+1)}\int_0^{\infty} f(x) x^{\alpha}e^{-x}dx.
\end{equation}
The linearization formula~\cite{Zeng92} reads as follows:
\begin{equation}\label{eq:lin}
{{\mathcal L}}(L^{(\alpha)}_{n_1}(x)L^{(\alpha)}_{n_2}(x)L^{(\alpha)}_{n_3}(x))= \sum_{s\geq
0}\frac{n_1!\,n_2!\,n_3!\,2^{n_1+n_2+n_3-2s}\,(\alpha+1)_s}
{(s-n_1)!\,(s-n_2)!\,(s-n_3)!\,(n_1+n_2+n_3-2s)!}.
\end{equation}

A combinatorial model for Laguerre polynomials with parameter $\alpha$ 
 was first  given by Foata and Strehl~\cite{FS84}. Recall that 
a \emph{Laguerre configuration} on $[n]:=\{1, \ldots, n\}$   is a pair 
$(A, f)$, where  $A\subset [n]$ and $f$ is an injection 
from $A$ to $[n]$.   
A Laguerre configuration can be depicted 
by a digraph on $[n]$ by drawing  an edge $i\to j$ if and only if 
$f(i)=j$. Clearly, such a graph has  two types of connected components called \emph{cycles} and \emph{paths}, see Figure~\ref{lagconf1}. 
 Let $\lc_{n,k}$ be the set of Laguerre configurations $(A, f)$ on $[n]$ 
with $|A|=n-k$. Then Foata and Strehl's interpretation~\cite{FS84} 
reads
\begin{align}\label{eq:FS}
\sum_{(A, f)\in \lc_{n,k}} (\alpha+1)^{\cyc(f)}=\frac{n!}{k!}\,
\binom{n+\alpha}{  n-k},
\end{align}
where $\cyc(f)$ is  the number of cycles of $f$.

\pagenumbering{arabic}
\addtocounter{page}{1}
\markboth{\SMALL QIONGQIONG PAN AND JIANG ZENG}{\SMALL COMBINATORICS OF
$(q,y)$-LAGUERRE POLYNOMIALS}

\begin{figure}[t]
\begin{tikzpicture}[scale=1.4]
\draw (0,0) rectangle (5.5,2.5);
\draw[line width=0.8,dashed] (4,2.5)--(4,0);
\draw[->] (0.5,1.5) to [out=-30, in=30] (0.55,0.52);
\draw[->] (0.5,0.5) to [out=120, in=210] (0.48,1.47);
\fill (0.5,1.5) circle (1.1pt);
\fill (0.5,0.5) circle (1.1pt);
\node[above] at (0.5,1.5) {$4$};
\node[below] at (0.5,0.5) {$7$};

\draw[->] (1.5,1.5) to [out=-30, in=60] (1.7,1);
\draw[->] (1.5,0.5) to [out=120, in=210] (1.45,1.5);
\draw[->] (1.7,0.95) to [out=-120,in=60] (1.5,0.55);
\fill (1.5,1.5) circle (1.1pt);
\fill (1.5,0.5) circle (1.1pt);
\fill (1.7,0.95) circle (1.1pt);
\node[above] at (1.5,1.5) {$2$};
\node[below] at (1.5,0.5) {$5$};
\node[right] at (1.7,0.95) {$13$};

\fill (2.5,1.5) circle (1.1pt);
\draw (2.5,1.5) to [out=-30,in=30] (2.5,1);
\draw [->](2.5,1) to [out=150,in=210] (2.48,1.45);
\node[above] at (2.5,1.5) {$14$};

\fill (1,2) circle (1.1pt);
\draw (1,2) to [out=-30,in=30] (1,1.5);
\draw [->](1,1.5) to [out=150,in=210] (0.98,1.95);
\node[above] at (1,2) {$15$};

\draw[->](3.6,1)--(3.6,0.55);
\draw[->](3.6,0.5)--(4.35,0.5);
\fill (3.6,1) circle (1.1pt);
\fill (3.6,0.5) circle (1.1pt);
\fill (4.4,0.5) circle (1.1pt);
\node[left] at (3.6,1){$12$};
\node[below] at (3.6, 0.5){$6$};
\node[right] at (4.4,0.5){$11$};

\draw[->] (3.6,1.4)--(4.3,1.4);
\fill (3.6,1.4) circle (1.1pt);
\fill (4.35,1.4) circle (1.1pt);
\node[left] at (3.6,1.4){$10$};
\node[right] at (4.4,1.4){$8$};

\draw[->] (3.6,2)--(4.3,2);
\fill (3.6,2) circle (1.1pt);
\fill (4.35,2) circle (1.1pt);
\node[left] at (3.6,2) {$1$};
\node[right] at (4.4,2) {$3$};

\fill (5,1.8) circle (1.1pt);
\node[right] at (5, 1.8){$9$};
\node[above] at (2,2.1){$\bf{A}$};
\end{tikzpicture}
\caption{A Laguerre configuration  $(A, f)$ on $[15]$ with 
$A=[15]\setminus \{3, 8, 9,11\}$.}
\label{lagconf1}
\end{figure}
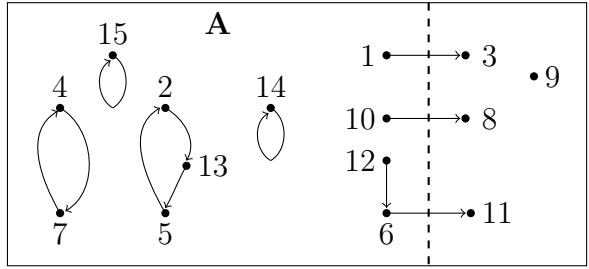

Note that one can derive \eqref{eq:FS} from any of the three formulas~\eqref{gf:laguerre}--\eqref{eq:recurr}, see \cite{FS84, BLL98}. 
The aim of this paper is 
to  study combinatorial aspects of more general 
 $(q,y)$-Laguerre polynomials 
$L_n^{(\alpha)}(x;y\,|\,q)$ ($n\geq 0$) defined  by 
 the three term-recurrence relation 
\begin{multline}\label{eq:recurrqlaguerre}
L^{(\alpha)}_{n+1}(x;y\,|\,q)=\left(x-(y[n+\alpha+1]_q+[n]_q)\right)L^{(\alpha)}_n(x;y\,|\,q)\\
-y[n]_q[n+\alpha]_q\,L^{(\alpha)}_{n-1}(x;y\,|\,q),
\qquad \alpha\geq -1, \; n\geq 1,
\end{multline}
with $L^{(\alpha)}_{0}(x;y\,|\,q)=1$, $L^{(\alpha)}_{-1}(x;y\,|\,q)=0$. 
Here  and  throughout this paper, we  use the standard $q$-notations:
$[n]_q=\frac{1-q^n}{1-q}$ for $n\geq 0$,
the $q$-analogue of $n$-factorial 
$n!_q=\prod_{i=1}^n[i]_q$, and the $q$-binomial coefficient
$$
\begin{bmatrix}n\\ k\end{bmatrix}_q=\frac{n!_q}{k!_q\, (n-k)!_q}\quad \textrm{for}\quad 0\leq k\leq n.
$$
Clearly we have $L^{(\alpha)}_{n}(x;1\,|\,1)=L^{(\alpha)}_{n}(x)$.   
Kasraoui et al.~\cite{KSZ11} gave a combinatorial interpretation for the linearization coefficients  of the polynomials $L^{(0)}_{n}(x;y\,|\,q)$ and pointed out that 
a combinatorial model for $L^{(0)}_{n}(x;y\,|\,q)$ can be derived from Simion 
and Stanton's model for octabasic $q$-Laguerre polynomials 
in \cite{SS96}.  For $k\in \Z$, let 
$$\N_k:=\{n\in \Z: n\geq k\}$$
and $\N:=\N_1$.
Recently, using the theory of  $q$-Riordan matrices,  Cheon, Jung and  Kim~\cite{CJK19}  derived a  combinatorial model for  the $q$-Laguerre polynomials 
$L^{(\alpha)}_{n}(x;q\,|\,q)$ when  $\alpha\in \N_0$.
  It is then natural to search for a   combinatorial structure unifying 
the above two special cases, as was alluded to at the end of \cite{CJK19}.
Our first goal is to give such a combinatorial model for $L^{(\alpha)}_{n}(x;y\,|\,q)$ with 
 variable $y$ and integer $\alpha\in \N_{-1}$  by using a $q$-analogue of 
Foata and Strehl's  Laguerre configurations.   
Moreover, for $\alpha\in \N_0$, the $(q,y)$-Laguerre polynomials 
$L_n^{(\alpha)}(x;y\,|\,q)$ are orthogonal polynomials.
It is our second goal
to give a combinatorial interpretation for the moments of $(q,y)$-Laguerre polynomials and prove that the linearization coefficients are polynomials in $y$ and $q$ with nonnegative integral coefficients.
We achieve this by making use of
the combinatorial theory of continued fractions.
  
By \eqref{eq:recurrqlaguerre}, the first few values of $L_n^{(\alpha)}(x;y\,|\,q)$ are 
\begin{align*}
L^{(\alpha)}_1(x;y\,|\,q)&=x-y[\alpha+1]_q,\\
L^{(\alpha)}_2(x;y\,|\,q)&=x^2-\left(y[\alpha+1]_q+y[\alpha+2]_q+1\right)x+[\alpha+1]_q[\alpha+2]_qy^2,\\
L^{(\alpha)}_3(x;y\,|\,q)&={x}^{3}-\left(y([\alpha+1]_q+[\alpha+2]_q+[\alpha+3]_q)+2+q)x^2\right.\\
&\kern1cm
+\left({y}^{2}([\alpha+1]_q[\alpha+2]_q+[\alpha+2]_q[\alpha+3]_q+[\alpha+1]_q[\alpha+3]_q)\right.\\
&\kern1cm
+\left.y([\alpha+3]_q+[2]_q[\alpha+1]_q)+[2]_q\right)x-y^3[\alpha+1]_q[\alpha+2]_q[\alpha+3]_q.
\end{align*}
For convenience, we  introduce the signless $(q,y)$-Laguerre polynomials
\begin{align}\label{lag-combinatrorial}
\L^{(\alpha)}_n(x;y\,|\,q):=(-1)^nL^{(\alpha)}_n(-x;y\,|\,q)=\sum_{k=0}^n  
\ell^{(\alpha)}_{n,k}(y;q)x^k.
\end{align}
For $\alpha\in \N_{-1}$, 
we observe that $\ell^{(\alpha)}_{n,k}(y;q)$  is a
polynomial in $y, q$ with nonnegative integral coefficients, which 
is far from  obvious from the explicit Formula~\eqref{eq:explaguerre}.
For $\alpha\in \N_{-1}$, Formula~\eqref{eq:FS} implies that  
$\ell^{(\alpha)}_{n,k}(1;1)$ is equal to the number of Laguerre configurations in $\lc_{n,k}$ such that each cycle 
carries a \emph{color} $\in [1+\alpha]$. 
In particular,  the number of   Laguerre configurations in $\lc_{n,k}$ without cycles 
(i.e., consisting of only $k$ paths)  is equal to the \emph{Lah numbers}~\cite{Lah55}:
$$
\ell^{(-1)}_{n,k}(1;1)=\frac{n!}{k!}\,\binom{n-1}{  k-1}.
$$
\begin{remark}
Two different   $q$-analogues of Lah numbers were  defined and studied 
 by Garsia and Remmel~\cite{GR80} and Lindsay et al.~\cite{LMS11}, respectively. 
 Moreover an elliptic analogue of Garsia and Remmel's $q$-Lah numbers was constructed 
 by Schlosser and Yoo~\cite{SY17}.
 \end{remark}

The organization of this paper is as follows. In Section~2 we identify the $(q,y)$-Laguerre polynomials as a rescaled version of  Al-Salam--Chihara polynomials and derive 
 several  expansion formulas  for $(q,y)$-Laguerre polynomials. 
 In Section~3 we present  
   a combinatorial interpretation for the  $(q,y)$-Laguerre polynomials 
   in terms of $\alpha$-Laguerre configurations, which are in essence 
    the product structure of ``cycles'' and ``paths''.
 In Section~4 we give 
 a combinatorial interpretation for  the moments of  $(q,y)$-Laguerre polynomials and  prove  that the linearization coefficients are polynomials  in $y$ and $q$ with nonnegative  integral coefficients. 
As the Laguerre polynomials play an important role
in the theory of rook polynomials, 
we translate  our  $\alpha$-Laguerre configurations in 
terms of rook placements in Section~5 and 
 set up the connection  
 between our 
$\alpha$-Laguerre configurations and the matching model  of \emph{complete bipartite graphs} $K_{n,n+\alpha}$ (see Godsil and Gutman~\cite{GG81}).

\section{A detour to Al-Salam--Chihara polynomials}

The $q$-Pochhammer symbol or $q$-shifted factorial $(a;q)_n$ is
 defined by
\[
    (a;q)_n= \begin{cases}
        \prod_{i=0}^{n-1}(1-aq^{i}) & \text{for } n\in \Z^+\cup\{\infty\},\\
        1& \text{for } n=0.
        \end{cases}
  \]
The Al-Salam--Chihara polynomials $Q_n(x):=Q_n(x;a,b\,|\,q)$ are defined 
by the  generating
function (see \cite[Chapter~14]{KLK10})
\begin{align}\label{gf-AC} 
\sum_{n=0}^\infty Q_n(x;a,b\,|\,q)\frac{t^n}{(q;q)_n}=
\frac{(at,bt;q)_\infty}{(te^{i\theta}, te^{-i\theta};q)_\infty},
\end{align}
with $(a,b;q)_\infty=(a;q)_\infty(b;q)_\infty$, and 
they satisfy  the recurrence relation~(op.\ cit.)
\begin{align}\label{eq:normalizedrecurr} 
\begin{cases}
Q_{-1}(x)=0,\quad Q_0(x)=1,\\
Q_{n+1}(x)=(2x-(a+b)q^n)Q_n(x)-
(1-q^n)(1-ab{q}^{n-1})Q_{n-1}(x),\quad n\geq 0.
\end{cases}
\end{align}
We have the explicit formula  
\begin{align}\label{formula1}
Q_{n}(x;a,b\,|\,q) =\frac{(ab;q)_{n}}{a^n}\,
\sum_{k=0}^n \frac{(q^{-n};q)_k\,(au;q)_k\,(au^{-1};q)_k}{(ab;q)_k(q;q)_k}q^k,
\end{align}
where $x=\frac{u+u^{-1}}2$ or $x=\cos \theta$ if $u=e^{i\theta}$.

 Comparing \eqref{eq:recurrqlaguerre} with  \eqref{eq:normalizedrecurr} 
 and using \eqref{lag-combinatrorial},
we see that our polynomials $\L^{(\alpha)}_n(x;y\,|\,q)$  are a
rescaled  version of the Al-Salam--Chihara polynomials:
\begin{equation}\label{eq:polydef}
\L^{(\alpha)}_n(x;y\,|\,q)=\left(\frac{\sqrt{y}}{1-q}\right)^n
Q_n\left(\frac{(1-q)x+y+1}{2\sqrt{y}}; \frac{1}{\sqrt{y}},
\sqrt{y}q^{\alpha+1}\,\bigg|\,q\right).
\end{equation}

The Al-Salam--Chihara polynomials  (see \cite[pp.~455--456]{KLK10} and \cite{Ism09}) are orthogonal with respect to the  linear functional
$\hat{\mathcal L}_q$ defined by 
\begin{align}
\hat{\mathcal L}_q(f)=\frac{(q, ab;q)_\infty}{2\pi}\; \int_{-1}^{+1}
\frac{f(x)dx}{\sqrt{1-x^2}}\prod_{k=0}^\infty 
\frac{1-2(2x^2 -1) q^k + q^{2k}}{[1-2x aq^k+ a^2q^{2k}][1-2x bq^k+ b^2q^{2k}]} .
\end{align}
Hence, for $\alpha\in \N_0$,  the polynomials $L^{(\alpha)}_n(x;y\,|\,q)$ are orthogonal with respect to the  linear functional
$\mathcal L_q$ given by
\begin{multline}\label{def:mom}
{\mathcal L}_{q}(f) = \frac{(q, q^{\alpha+1};q)_\infty }{2\pi}\frac{1-q}{2\sqrt{y}}\; \int_{B_{-}}^{B_{+}}
 \frac{f(x)dx }{\sqrt{1-v(x)^2}}   \\
\times \prod_{k=0}^\infty\frac{[1-2(2v(x)^2-1) q^k + q^{2k}]}{[1-2v(x)q^k/\sqrt{y}+ q^{2k}/y]
[1-2v(x)  q^{k+\alpha+1}\sqrt{y} + q^{2k+2\alpha+2}y]},
\end{multline}
where $B_{\pm} = \frac{(1\pm \sqrt{y})^2}{1-q} $ and 
\begin{align}
v(x) = \frac{1}{2\sqrt{y}} ((q-1)x +(y+1)).
\end{align}
Now,  by \eqref{eq:polydef}, we may derive an explicit formula from \eqref{formula1}, namely
 \begin{equation}\label{eq:explaguerre}
\L^{(\alpha)}_n(x;y\,|\,q)=\sum_{k=0}^n \frac{n!_q}{k!_q}\,\begin{bmatrix}n+\alpha\\ k+\alpha\end{bmatrix}_q
q^{k(k-n)}y^{n-k} \prod_{j=0}^{k-1}\left(x+(1-y q^{-j})[j]_q\right),
\end{equation}
and, from \eqref{gf-AC}, the generating function 
\begin{align}
\mathcal L^{(\alpha)}(x;y;t\,|\,q)&:=\sum_{n\geq 0} \L^{(\alpha)}_n(x;y\,|\,q)\frac{t^n}{n!_q}\nonumber\\
&=
\frac{(t;q)_\infty\,(ytq^{\alpha+1};q)_\infty}{\prod_{k=0}^\infty \left[1-((1-q)x+y+1)tq^k+yt^2q^{2k}\right]},\label{GF}
\end{align}
which can be written as 
\begin{align}\label{key2.5}
\mathcal L^{(\alpha)}(x;y;t\,|\,q)=\mathcal L^{(\alpha)}(0;y;t\,|\,q)\cdot
\mathcal L^{(-1)}(x;y;t\,|\,q).
\end{align}

Define the ``vertical generating function''
 \begin{align}
 \mathcal L^{(\alpha)}_k(y;t\,|\,q):=[x^k]\mathcal L^{(\alpha)}(x;y;t\,|\,q)=\sum_{n\geq k}
 \ell^{(\alpha)}_{n,k}(y,q)\frac{t^n}{n!_q},
\end{align}
and  the $q$-derivative operator $\mathcal{D}_q$  
for  $f(t)\in \mathbf{R}[[t]]$  by 
$$
\mathcal{D}_q(f(t))=\frac{f(t)-f(qt)}{(1-q)t},
$$
where $\mathbf{R}=\C[[x,y,q,\ldots]]$.
Thus $\mathcal{D}_q(1)=0$ and $\mathcal{D}_q(t^n)=[n]_qt^{n-1}$ for $n>0$.

It follows from \eqref{GF} that 
\begin{align}\label{der-1}
\mathcal{D}_q\mathcal L^{(-1)}(x;y;t\,|\,q)=\frac{x}{(1-t)(1-yt)}\mathcal L^{(-1)}(x;y;t\,|\,q),
\end{align}
which in particular gives
\begin{align}
 \mathcal D_q\mathcal L_1^{(-1)}(y;t\,|\,q)&=[x]\mathcal{D}_q\mathcal L^{(-1)}(x;y;t\,|\,q)\nonumber\\
 &=\frac{1}{(1-t)(1-ty)}\nonumber\\
 &=\sum_{n\geq 0}n!_q\,[n+1]_y\,\frac{t^n}{n!_q}.\label{-1}
 \end{align}
So we can rewrite \eqref{der-1} as 
%
\begin{align}\label{prop1bis}
\mathcal D_q\mathcal L^{(-1)}(x;y;t\,|\,q)=x\cdot 
\mathcal D_q\mathcal L_1^{(-1)}(y;t\,|\,q)\cdot \mathcal L^{(-1)}(x;y;t\,|\,q),
\end{align}
which is equivalent to  the following result.
\begin{prop}\label{G}
For $n\in \N$, we have
\begin{align}\label{eq:a=-1}
\L_{n+1}^{(-1)}(x;y\,|\,q)=x\sum_{k=0}^n\begin{bmatrix}n\\ k\end{bmatrix}_qk!_q\,[k+1]_y\,\L_{n-k}^{(-1)}
(x;y\,|\,q).
\end{align}
\end{prop}

Now,  applying  the $q$-binomial formula (see \cite[Chapter~1]{Gas:Rah04})
$$
\sum_{n\geq 0}\frac{(a;q)_n}{(q;q)_n}z^n=\frac{(az;q)_\infty}{(z;q)_\infty}
$$
with $a=q^{\alpha+1}$ and $z=yt$, we have 
\begin{align}\label{val.l0}
\mathcal L^{(\alpha)}(0;y;t\,|\,q)
=\frac{(ytq^{\alpha+1};q)_\infty}{(yt;q)_\infty}=
\sum_{n\geq 0} \bigg(\prod_{k=1}^n[\alpha+k]_q\bigg)\frac{(yt)^n}{n!_q}.
\end{align}
Substitution of the latter into  \eqref{key2.5} 
gives the following result.
\begin{prop} For $n\in \N$, we have 
\begin{align}\label{connection-1}
\L^{(\alpha)}_n(x;y\,|\,q)=\sum_{k=0}^n \begin{bmatrix}n\\ k\end{bmatrix}_q 
\bigg(\prod_{j=1}^k[\alpha+j]_q\bigg)
y^{k}\L^{(-1)}_{n-k}(x;y\,|\,q).
\end{align}
\end{prop}


 \begin{remark}
 \begin{enumerate}
 \item More generally we can prove 
  the following connection formula 
  for $\alpha\geq \beta\geq -1$:
\begin{align}\label{connectionbeta}
\L^{(\alpha)}_n(x;y\,|\,q)=\sum_{k=0}^n \begin{bmatrix}n\\ k\end{bmatrix}_q 
\bigg(\prod_{j=0}^{k-1}[\alpha-\beta+j]_q\bigg)
(yq^{\beta+1})^{k}\L^{(\beta)}_{n-k}(x;y\,|\,q).
\end{align}
\item  For $q\to 1$, Identity~\eqref{GF} reduces to
\begin{align*}
\sum_{n\geq 0}\L^{(\alpha)}_n(x;y\,|\,1)
 \frac{t^n}{n!}
 &=(1-yt)^{-(\alpha+1)}\left( 1- \frac{(1-y)t}{1-yt} \right)^{-x/(1-y)}.
\end{align*}
Comparing with the generating function of the Meixner polynomials (see~\cite[Equation~(1.9.11)]{KLK10})
$$
\sum_{n=0}^\infty \frac{(\beta)_n}{n!} M_n(x;\beta,c) t^n=(1-t)^{-x-\beta}(1-t/c)^x,
$$
we derive 
\begin{align*}
\L^{(\alpha)}_n(x;y\,|\,1)=y^n(\alpha+1)_n  M_n\left(\frac{-x}{1-y}; \alpha+1, y\right).
\end{align*}
Hence the $(q,y)$-Laguerre polynomials $\L^{(\alpha)}_n(x;y\,|\,q)$ are a $q$-analogue of rescaled Meixner polynomials.
\end{enumerate}
\end{remark}
\section{Combinatorial interpretation of $(q,y)$-Laguerre polynomials}
The reader is  referred  to \cite{GJ83, Fla:Seg09, BLL98} for  the general combinatorial theory of exponential 
generating functions for labeled structures. For our purpose we need only 
a $q$-version of this theory  for special labeled structures.
A \emph{labeled structure} on 
a (finite) set $A\subset \N$ is a graph with vertex set $A$.  
Consider a family of labeled $\F$-structures $\F=\bigcup_{n=0}^\infty \F_n$, where 
$\F_n$ consists of  the $\F$-structures on $[n]$.  
If $A=\{a_1, \ldots, a_n\}\subset \N$, where $a_1<\cdots <a_n$, 
an $\F$-structure  on $A$ is obtained by replacing $i$ by $a_i$ for $i=1, \ldots, n$ in the elements of $\F_n$.
Let $\mathcal{F}[A]$ denote the set of $\mathcal{F}$-structures on $A$ and 
associate a weight $u(f)$ to each object $f\in \F$.
 For the set of weighted  $\mathcal{F}$-structures  $\mathcal{F}_u$~(where the valuation $u$ may involve the parameter $q$), the $q$-generating function is defined as 
$$
\mathcal{F}_u(t)= \sum_{f\in\mathcal{F}}u(f)\frac{t^{|f|}}{|f|!_q},
$$
where $|f|=n$ if $f\in \mathcal{F}[n]$.
 If $\mathcal{F}_u$ and $\mathcal{G}_v$ are two         
 weighted  structures, we denote by 
 $(\mathcal{F}\cdot\mathcal{G})_w[n]$ the set of
   pairs $(f,g)\in\mathcal{F}[S]\times\mathcal{G}[T]$ with weight  
$$
w(f,g)=u(f)\cdot v(g)\cdot q^{\textrm{inv}(S,T)},
$$
where
   $(S,T)$ is an ordered bipartition of $[n]$ 
   and  $\textrm{inv}(S,T)$ is  the number of pairs $(i,j)\in S\times T$ such that $i>j$. Recall
(see \cite[p.~98]{GJ83}) that 
$$
\sum_{(S,T)}q^{\textrm{inv}(S,T)}=\begin{bmatrix}n\\ k\end{bmatrix}_q,
$$
where the sum is over all ordered bipartitions $(S,T)$ of $[n]$ with $|S|=k$.
It is folklore and immediately checked that
\begin{align}\label{folklore}
(\mathcal{F}\cdot\mathcal{G})_w(t)=\mathcal{F}_u(t)\cdot
\mathcal{G}_v(t).
\end{align}
We need some further definitions.
\begin{itemize}
\item[(a)] For a permutation $\sigma$ of a set $A\subset\mathbb{N}$, let the word $\hat{\sigma}$ denote its linear representation in the usual sense, i.e.,  $\hat{\sigma}=\sigma(i_1)\ldots \sigma(i_n)$ if $A=\{i_1, \ldots, i_n\}$ with $i_1<\cdots <i_n$.
\item[(b)]  
A list of~(nonnegative) integers, taken as a word over $\mathbb{N}$, is \textit{strict} if no element occurs more than once.
For a strict list $\rho$ let $\rl(\rho)$ be the number of elements that come after the maximum element.
\item[(c)]  
For a set $\lambda$ of $k$ non-empty and disjoint strict lists of integers, order these lists according to their minimum element~(increasing). This gives a list of $k$
words $(\lambda_1,\ldots,\lambda_k)$, which will be identified with $\lambda$. Then $\underline{\lambda}=\lambda_1\ldots\lambda_k$ denotes the concatenation of these lists.
\end{itemize}

Two particular structures  will be used to interpret the $(q,y)$-Laguerre polynomials.

\begin{itemize}
\item[(d)] The structures $\mathcal{S}^{(\alpha)}$  consist of permutations $\sigma$, where each cycle carries a color $\in\{0,1,2,\ldots,\alpha\}$. Write $\sigma$ as a product of \emph{unicolored permutations}, $\sigma=\sigma_0\cdot\sigma_1\cdots\sigma_{\alpha}$, where $\sigma_i$ is the product  of cycles with color $i$. Now  consider the concatenation
$$
\underline{\sigma}=\hat{\sigma}_0\cdot\hat{\sigma}_1\cdots
\hat{\sigma}_{\alpha}
$$
and the word with letters from $\{0, 1\}$ given by
$$
\underline{\underline{\sigma}}=0^{|\hat{\sigma}_0|}10^{|\hat{\sigma}_1|}1\cdots 10^{|\hat{\sigma}_\alpha|}.
$$
Define the valuation $u$ on $\mathcal{S}^{(\alpha)}$ by
$$
u(\sigma)=y^{|\underline{\sigma}|}q^{\inv(\underline{\sigma})+\inv(\underline{\underline{\sigma}})}.
$$
\item[(e)] The structures $\Lin^{(k)}$ consist of sets $\lambda=(\lambda_1,\ldots,\lambda_k)$ of $k$ nonempty and disjoint strict lists~(cf.~(c)). Define the valuation $v$ on $\Lin^{(k)}$ by
$$
v(\lambda)=y^{\rl(\lambda)}q^{\inv(\underline{\lambda})-\rl(\lambda)},
$$
where $\rl(\lambda)=\sum_{i=1}^k\rl(\lambda_i)$.
\end{itemize}

Let $ \lc^{(\alpha)}_{n,k}:=\mathcal{S}^{(\alpha)}\cdot \Lin^{(k)}[n]$. 
For  any $\alpha$-Laguerre configuration $(\sigma,\lambda)\in\mathcal{S}^{(\alpha)}[A]\times \Lin^{(k)}[B]$~with $A\cap B=\emptyset$ and $A\cup B=[n]$, in order to invoke the folklore statement~\eqref{folklore}, one should use as valuation 
\begin{align}
w(\sigma,\lambda)&=u(\lambda)\cdot v(\lambda)\cdot q^{\inv(A,B)}\nonumber\\
&=y^{|\underline{\sigma}|}q^{\inv(\underline{\sigma})+
\inv(\underline{\underline{\sigma}})}y^{\rl(\lambda)}q^{\inv(\underline{\lambda})-\rl(\lambda)}q^{\inv(A,B)}\nonumber\\
&=y^{|\underline{\sigma}|+\rl(\lambda)}q^{\inv(\underline{\sigma})+
\inv(\underline{\underline{\sigma}})+\inv(\underline{\lambda})
-\rl(\lambda)}q^{\inv(A,B)}\nonumber\\
&=y^{|\underline{\sigma}|+\rl(\lambda)}q^{\inv(\underline{\sigma}.\underline{\lambda})-\rl(\lambda)+\inv(\underline{\underline{\sigma}})}.\label{w-valuation}
\end{align}
The essential point is 
$\inv(\underline{\sigma})+\inv(\underline{\lambda})+\inv(A,B)=
\inv(\underline{\sigma}.\underline{\lambda})$.
This describes the weighted configurations $( \lc^{(\alpha)}_{n,k})_w$. 
An element of   $(\lc^{(\alpha)}_{n,k})_w$ is called an 
$\alpha$-Laguerre configuration  on $[n]$, see Figure~\ref{lagconf2}.
\begin{figure}[t]
\begin{tikzpicture}[scale=1.4]
\draw (0,0) rectangle (5.5,2.5);
\draw[line width=0.8,dashed] (3,2.5)--(3,0);
\draw[->] (0.5,1.5) to [out=-30, in=30] (0.55,0.52);
\draw[->] (0.5,0.5) to [out=120, in=210] (0.48,1.47);
\fill (0.5,1.5) circle (1.1pt);
\fill (0.5,0.5) circle (1.1pt);
\node[above] at (0.5,1.5) {$4$};
\node[below] at (0.5,0.5) {$7$};
\node[above,color=red] at (0.5,0.7){\bf{0}};

\draw[->] (1.5,1.5) to [out=-30, in=60] (1.7,1);
\draw[->] (1.5,0.5) to [out=120, in=210] (1.45,1.5);
\draw[->] (1.7,0.95) to [out=-120,in=60] (1.5,0.55);
\fill (1.5,1.5) circle (1.1pt);
\fill (1.5,0.5) circle (1.1pt);
\fill (1.7,0.95) circle (1.1pt);
\node[above] at (1.5,1.5) {$2$};
\node[below] at (1.5,0.5) {$5$};
\node[right] at (1.7,0.95) {$13$};
\node[above,color=red] at (1.5,0.7)   {\bf{1}};

\fill (2.5,1.5) circle (1.1pt);
\draw (2.5,1.5) to [out=-30,in=30] (2.5,1);
\draw [->](2.5,1) to [out=150,in=210] (2.48,1.45);
\node[above] at (2.5,1.5) {$14$};
\node[above,color=red] at (2.5,1) {\bf{1}};

\fill (1,2) circle (1.1pt);
\draw (1,2) to [out=-30,in=30] (1,1.5);
\draw [->](1,1.5) to [out=150,in=210] (0.98,1.95);
\node[above] at (1,2) {$15$};
\node[above,color=red] at (1,1.5) {\bf{0}};

\draw[->](3.6,1)--(3.6,0.55);
\draw[->](3.6,0.5)--(4.35,0.5);
\fill (3.6,1) circle (1.1pt);
\fill (3.6,0.5) circle (1.1pt);
\fill (4.4,0.5) circle (1.1pt);
\node[left] at (3.6,1){$12$};
\node[below] at (3.6, 0.5){$6$};
\node[right] at (4.4,0.5){$11$};

\draw[->] (3.6,1.4)--(4.3,1.4);
\fill (3.6,1.4) circle (1.1pt);
\fill (4.35,1.4) circle (1.1pt);
\node[left] at (3.6,1.4){$10$};
\node[right] at (4.4,1.4){$8$};

\draw[->] (3.6,2)--(4.3,2);
\fill (3.6,2) circle (1.1pt);
\fill (4.35,2) circle (1.1pt);
\node[left] at (3.6,2) {$1$};
\node[right] at (4.4,2) {$3$};

\fill (5,1.8) circle (1.1pt);
\node[right] at (5, 1.8){$9$};
\node[above] at (2,2.1){$\sigma$};
\node[above] at (3.5,2.1){$\lambda$};
\end{tikzpicture}
\caption{A $1$-Laguerre configuration  $(\sigma, \la)\in  \lc^{(1)}_{15,4}$, which is  the Laguerre configuration in Figure~\ref{lagconf1} of which each cycle gets a color $0$ or $1$.}
\label{lagconf2}
\end{figure}
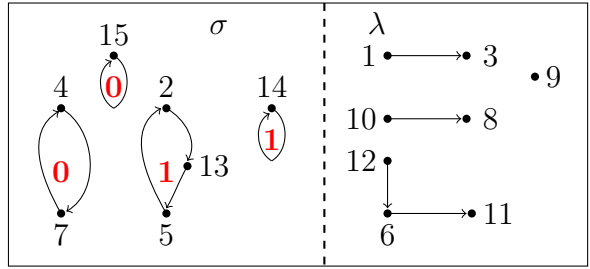

\begin{lem}\label{lemma1}
For $\alpha\in\mathbb{N}$, we have 
$$
\mathcal{S}_u^{(\alpha)}(t)=
\mathcal{L}^{(\alpha)}(0;y;t\,|\,q).
$$
\end{lem}
\begin{proof}
Let $\mathsf{P}(n,\alpha)$ be the set of words of length $n+\alpha$ 
with $n$ $0$'s and $\alpha$ $1$'s, i.e., lattice paths from $(0,0)$ to $(n,\alpha)$.
For $\sigma\in\mathcal{S}^{(\alpha)}[n]$,
the word $\underline{\sigma}$ can be seen as the linear representation of an~(ordinary) permutation $\tilde{\sigma}\in\mathcal{S}^{(0)}[n]$, whereas 
$\underline{\underline{\sigma}}\in \mathsf{P}(n,\alpha)$. The mapping
\begin{align*}
\mathcal{S}^{(\alpha)}[n] &\rightarrow \mathcal{S}^{(0)}[n]\times \mathsf{P}(n,\alpha)\\ \sigma &\mapsto (\tilde{\sigma},\underline{\underline{\sigma}})
\end{align*}
is a bijection, and from summing both contributions separately, one obtains
\begin{align*}
\sum_{\sigma\in\mathcal{S}^{(\alpha)}[n]}q^{\inv(\underline{\sigma})+\inv(\underline{\underline{\sigma}})}&=
\sum_{\sigma\in\mathcal{S}^{(0)}[n]}q^{\inv(\underline{\sigma})}
\sum_{\underline{\underline{\sigma}}\in P(n,\alpha)}q^{\inv(\underline{\underline{\sigma}})}\\
&=n!_q\begin{bmatrix}n+\alpha\\\alpha\end{bmatrix}_q,
\end{align*}
which is $\prod_{i=1}^n[\alpha+i]_q$. So we get 
$$
\mathcal{S}_u^{(\alpha)}(t)=\sum_{n\geq 0}
\bigg(\prod_{i=1}^n[\alpha+i]_q\bigg)\frac{(yt)^n}{n!_q}.
$$
The result then follows from~\eqref{val.l0}.
\end{proof}
\begin{lem}\label{lemma2}
For integers $k\geq 1$, we have
$$
\Lin_v^{(k)}(t)=\mathcal{L}_k^{(-1)}(y;t\,|\,q).
$$
\end{lem}
\begin{proof} We proceed by induction on $k\geq 1$.
\begin{itemize}
\item The case  $k=1$. 
For a single list $\lambda=\underline{\lambda}\in \Lin^{(1)}[n+1]$, let $j_{\lambda}$ be the position of the maximum element, let $\lambda'=\underline{\lambda}'\in \Lin^{(1)}[n]$ be the list obtained by deleting this maximum element. Then
\begin{align*}
\Lin^{(1)}[n+1] &\rightarrow \Lin^{(1)}[n]\times[n+1]\\
\lambda&\mapsto(\lambda',j_{\lambda})
\end{align*}
is a bijection such that 
$\inv(\underline{\lambda})=\inv(\underline{\lambda}')+
\rl(\lambda)$.
Furthermore, we have
\begin{align*}
\sum_{\lambda\in \Lin^{(1)}[n+1]}y^{\rl(\lambda)}q^{\inv(\underline{\lambda})-\rl(\lambda)}&=
\sum_{\lambda'\in \Lin^{(1)}[n]}q^{\inv(\underline{\lambda'})}\sum_{j\in[n+1]}
y^{n+1-j},
\end{align*}
and thus
$$
\sum_{\lambda\in \Lin^{(1)}[n+1]}v(\lambda)=n!_q\,[n+1]_{y},
$$
which, in view of \eqref{-1},  gives
$$
\mathcal{D}_q\Lin_v^{(1)}(t)=\mathcal{D}_q\mathcal{L}_1^{(-1)}(y;t\,|\,q),
$$
and by $q$-integration
$$
\Lin_v^{(1)}(t)=\mathcal{L}_1^{(-1)}(y;t\,|\,q)
$$
because the series on both sides have a zero constant term.

\item The case $k>1$.
Assuming that
$\Lin_v^{(k)}(t)=\mathcal{L}_k^{(-1)}(y;t\,|\,q)$ has already been proved for  $k\geq 1$, 
the goal is to show
$$
\Lin_v^{(k+1)}(t)=\mathcal{L}_{k+1}^{(-1)}(y;t\,|\,q).
$$
Comparing  the coefficients of  $x^{k+1}$ on both sides of Equation~\eqref{prop1bis},
we obtain
$$
\mathcal{D}_q\mathcal{L}^{(-1)}_{k+1}(y;t\,|\,q)=\mathcal{D}_q\mathcal{L}^{(-1)}_1(y;t\,|\,q)\cdot \mathcal{L}^{(-1)}_k(y;t\,|\,q).
$$
If we can show that similarly
\begin{align}\label{diffeq}
\mathcal{D}_q\Lin_v^{(k+1)}(t)=\mathcal{D}_q\Lin_v^{(1)}(t)\cdot \Lin_v^{(k)}(t),
\end{align}
then we would be done. Again, the final integration step poses no problem because in both $\Lin_v^{(k+1)}(t)$ and $\mathcal{L}_{k+1}^{(-1)}(y;t\,|\,q)$ the first $k+1$ coefficients vanish.
Recall that a configuration 
$\la\in \Lin^{(k+1)}[n]$  
consists of a list  of $k+1$ disjoint strict lists, written as a list $\lambda=(\lambda_0,\lambda_1,\ldots,\lambda_k)$,
with $\lambda_i\in \Lin^{(1)}[A_i]$, where
$$
\biguplus_{i=0}^kA_i=[n]\quad \textrm{and}\quad \min A_{i-1}<\min A_i,\quad 
1\leq i\leq k.
$$
We have  a bijection
\begin{align*}
\Lin^{(k+1)}[A] &\rightarrow \Lin^{(1)}[A_0]\times \Lin^{(k)}[A']\\
\lambda&\mapsto(\lambda_0,\lambda'),
\end{align*}
where $\lambda'=(\lambda_1,\dots,\lambda_k)$ and $A'=\bigcup_{i=1}^kA_i$, which also satisfies the requirement for applying the folklore statement~\eqref{folklore}:
$$
v(\lambda)=v(\lambda_0)\cdot v(\lambda')\cdot q^{\inv(A_0,A')}.
$$
All this holds only if for the bipartition $A=A_0\uplus A'$ it is guaranteed that $\min A_0<\min A'$. This is where the derivative $\mathcal{D}_q$ comes into play. Differentiation for a collection of structures means that the minimum element of the underlying set of a structure is tagged and no longer counted in the $w$-valuation of the base set. In the present situation, this implies that only structures are considered where tagging the minimum element of $\lambda$ means the same as tagging the minimum element of $\lambda_0$. This shows that \eqref{diffeq}
holds.\qedhere
\end{itemize}
\end{proof}

\begin{thm}\label{key}
For integers $\alpha\geq -1$,
we have
\begin{align*}
\ell_{n,k}^{(\alpha)}(y;q)
&=\sum_{(\sigma;\lambda)\in\lc_{n,k}^{(\alpha)}}y^{|\underline{\sigma}|+\rl(\lambda)}q^{\inv(\underline{\sigma}.\underline{\lambda})-\rl(\lambda)+\inv(\underline{\underline{\sigma}})}.
\end{align*}
\end{thm}
\begin{proof}
From \eqref{key2.5} we infer
$$
\mathcal L_k^{(\alpha)}(y;t\,|\,q)=\mathcal L^{(\alpha)}(0;y;t\,|\,q)\mathcal L_k^{(-1)}(y;t\,|\,q),
$$
and the result  follows from Lemmas~\ref{lemma1} and~\ref{lemma2}. 
\end{proof}

Here we give an example to illustrate the $\alpha$-Laguerre configurations.
\begin{example}\label{example1}
Consider  the $1$-Laguerre configuration  
$(\sigma; \la)\in \mathcal{LC}^{(1)}_{15,4}$ in Figure~\ref{lagconf2}.
We have 
\begin{align*}
\sigma&=\sigma_0\cdot \sigma_1\quad \text{with}\quad \sigma_0=(15)(7\,4),\;\sigma_1=(14)(13\,5\,2);\\
\lambda&=(\lambda_1,\lambda_2,\lambda_3,\lambda_4)\quad \text{with}\quad
\lambda_1=1\,3,\;\lambda_2=12\,6\,11,\; \lambda_3=10\,8,\;\lambda_4=9.
\end{align*}
Thus, 
\begin{align*}
&\underline{\sigma}=\hat\sigma_0\cdot \hat\sigma_1=7\,4\,15\cdot13\,2\,5\,14,\\
&\underline{\underline{\sigma}}=0^3\,1\,0^4;\\
&\underline{\la}=1\,3\cdot 12\,6\,11\cdot 10\,8\cdot 9.
\end{align*}
We have $|\underline{\sigma}|=7$, $\rl(\la)=3$, $\inv(\underline{\underline{\sigma}})=4$, 
and $\inv(\underline{\sigma}\cdot \underline{\la})=52$.
\end{example}
\begin{remark} Our model of $\alpha$-Laguerre configurations 
is simpler than the model in~\cite{CJK19}.  
Actually, the $\alpha$-Laguerre configurations are essentially 
the Laguerre configurations of which each 
cycle has a color in $\{0, \ldots, \alpha\}$. 
A linear order of paths and colored cycles is needed
only  for  the valuation $w$ in \eqref{w-valuation}.
\end{remark}

\section{Moments of $(q,y)$-Laguerre polynomials}
For $\alpha\in \N_0$,  by \eqref{def:mom} the moments of
the $(q,y)$-Laguerre polynomials are defined by
\begin{align}
\mu_n^{(\alpha)}(q,y):={\mathcal L}_{q}(x^n).
\end{align}

According to the theory of orthogonal polynomials (see \cite{Chi78}) and 
the three-term recurrence relation~\eqref{eq:recurrqlaguerre},
we have the orthogonality  relation
\begin{align}\label{eq:orthog}
{\mathcal L}_q(L^{(\alpha)}_{n}(x;y\,|\,q)L^{(\alpha)}_{m}(x;y\,|\,q))=
y^{n}n!_q\bigg(\prod_{j=1}^n[\alpha+j]_q\bigg)\delta_{n\,m}.
\end{align}
Moreover, we have 
 the  following continued fraction
expansion:
\begin{equation}\label{eq:co}
\sum_{n\geq
0}\mu_n^{(\alpha)}(q,y)t^n=\frac{1}{1-b_0t-\displaystyle
\frac{\lambda_1 t^2}{1-b_1t-\displaystyle\frac{\lambda_2
t^2}{\ddots}}},
\end{equation}
where $b_n=y[n+\alpha+1]_q+[n]_q$ and $\lambda_n=y[n]_q[n+\alpha]_q$.

 Let $\S_n$ be  the set of permutations of $\{1,2,\dots,n\}$. 
  For  $\sigma\in \S_n$, 
  we define three statistics, namely:
  \begin{itemize}
 \item the  number of \emph{weak excedances}, $\wex(\sigma)$, given by
 \begin{align*}
 \wex(\sigma)=|\{i\in [n]: \sigma(i)\geq i\}|;
 \end{align*}
\item the  number of \emph{records} (or \emph{left-to-right maxima}), $\rec(\sigma)$, given by
\begin{align*}
 \rec(\sigma)=|\{i\in [n]: \sigma(i)>\sigma(j) \;\textrm{for all}\;  j<i\}|;
  \end{align*}
 \item the \emph{number of crossings}, $\cros(\sigma)$, given by 
 \begin{align*}
\cros(\sigma)=|\{(i,j)\in [n]\times [n]: i<j\leq \sigma(i)<\sigma(j)\;\textrm{or}\; 
 \sigma(j)<\sigma(i)<j<i \}|.
\end{align*}
\end{itemize}

\begin{thm}\label{mom} Let $\beta=[\alpha+1]_q$. Then
\begin{align}\label{mom2}
\mu_n^{(\alpha)}(y,q)=\sum_{\sigma\in
\S_n}\beta^{\rec(\sigma)}y^{\wex(\sigma)}q^{\cros(\sigma)}. 
\end{align}
\end{thm}

The first values of the moments  are as follows:
\begin{align*}
\mu_1^{(\alpha)}(y,q)&=y\beta,\\
\mu_2^{(\alpha)}(y,q)&=y\beta+y^2\beta^2,\\
\mu_3^{(\alpha)}(y,q)&=y\beta+\beta(1+(2+q)\beta)y^2+y^3\beta^3.
\end{align*}

Due to the contraction formula~\cite[p.~292]{GJ83}, we can rewrite \eqref{eq:co} as
\begin{equation}\label{eq:co-stieltjes}
\sum_{n\geq
0}\mu_n^{(\alpha)}(q,y)t^n=\frac{1}{1-\displaystyle
\frac{\gamma_1 t}{1-\displaystyle\frac{\gamma_2
t}{\ddots}}},
\end{equation}
where $\gamma_{2n}=[n]_q$ and 
$\gamma_{2n+1}=y[n+\alpha]_q=y([n]_q+[\alpha+1]_q q^n)$ for $n\geq 0$.


Recall that 
a \emph{Dyck path} of length $2n$ is a sequence of points $(\omega_0, \ldots, \omega_{2n})$ in $\N_0\times \N_0$ 
satisfying $\omega_0=(0,0), \;\omega_{2n}=(2n,0)$ and 
$\omega_{i+1}-\omega_i=(1, 1)\;\textrm{or}\; (1,-1)$ for $i=0, \ldots, 2n-1$. Clearly we can also  identify a Dyck path with its sequence of steps (or \emph{Dyck word})
$s=s_1\ldots s_{2n}$ on the alphabet  $\{\mathsf{u}, \mathsf{d}\}$, and we use $|s|_{\mathsf{u}}$ and $|s|_{\mathsf{d}}$ to denote the number of $\mathsf{u}$'s and $\mathsf{d}$'s, respectively, in $s$.
So, for a Dyck word $s$, we have $|s|_{\mathsf{u}}= |s|_{\mathsf{d}}=n$ and  
$|s_1\ldots s_k|_{u}\geq |s_1\ldots s_k|_{\mathsf{d}}$ for 
$k\in [2n]$. The height $h_k$ of step $s_k$ is defined to be $h_1=0$ and 
$$
h_k=|s_1\ldots s_{k-1}|_u-|s_1\ldots s_{k-1}|_d \quad\textrm{for}\quad  k=2,\ldots, 2n.
$$
A  \emph{Laguerre history} of length $2n$ is a pair $(s, \xi)$, where 
$s$ is a Dyck word of length $2n$ and 
$\xi=(\xi_1, \ldots, \xi_{2n})$ is a sequence of integers such that 
$\xi_i=1$ if $s_i=\mathsf{u}$ and $1\leq \xi_i\leq \lceil{h_i/2\rceil}$ if $s_i=\mathsf{d}$.
Let $\mathcal{LH}_n$ be the set of  Laguerre histories of length $2n$. We essentially   use  Biane's  bijection~\cite{Bi93} to 
construct a bijection $\Phi$ from $\S_n$ to $\mathcal{LH}_n$. 

\begin{proof}[Proof of Theorem~\ref{mom}] 
We identify  a permutation $\sigma\in \S_n$ with the   bipartite graph $\mathcal G$ on 
$\{1, \ldots, n; 
1', \ldots, n'\}
$ with  an edge  $(i,j')$ 
if and only if $\sigma(i)=j$. 
We display the vertices on  two rows called top row and bottom row  as follows:
$$
\begin{pmatrix}
1&2&\cdots&n\\
1'&2'&\cdots &n'
\end{pmatrix},
$$
and we read the graph column by column from left to right and from top to bottom.
In other words,  the order of vertices   is $v_1=1, v_2=1',  \ldots, v_{2n-1}=n,
v_{2n}= n'$. 

For $k=1, \ldots, 2n$,  the $k$-th restriction  
of $\mathcal G$ is  the graph  $\mathcal G_k$ on $\{v_1, v_2, \ldots, v_k\}$
with edge $(v_i, v_j)$ in $\mathcal G_k$  if and only if $i,j\in [k]$, so  isolated
vertices may exist in $\mathcal G_k$. 

For $i=1, \ldots, n$, the Dyck path   $s=s_1\ldots s_{2n}$ is defined as follows:
\begin{itemize}
\item if $\sigma^{-1}(i)>i<\sigma(i)$ (i.e., $i$ is a cycle valley), 
then  $s_{2i-1}s_{2i}=\mathsf{uu}$;
\item if  $\sigma^{-1}(i)< i<\sigma(i)$ 
(i.e., $i$ is a cycle double ascent), then 
$s_{2i-1}s_{2i}=\mathsf{ud}$;
\item if $\sigma^{-1}(i)>i>\sigma(i)$ (i.e., $i$ is a cycle double descent),
 then  $s_{2i-1}s_{2i}=\mathsf{du}$;
\item if $\sigma^{-1}(i)<i>\sigma(i)$ (i.e., $i$ is a cycle peak), then 
$s_{2i-1}s_{2i}=\mathsf{dd}$;
\item if  $\sigma^{-1}(i)= i=\sigma(i)$ 
(i.e., $i$ is a fixed point), then 
$s_{2i-1}s_{2i}=\mathsf{ud}$.
\end{itemize}
It is easy to see that 
\begin{itemize}
\item $s$ is a Dyck path; 
\item the height $h_i$ is the number of isolated vertices in $\mathcal G_{i-1}$
for $i\in [2n]$ with $\mathcal G_{i-1}=\emptyset$;
thus $h_{2i-1}$ (respectively $h_{2i}$) is even (respectively odd)  for $i=1, \ldots, n$ and 
 there are $\lceil{h_i/2}\rceil$ isolated vertices in the top row. 
 \end{itemize}
Next, the sequence $\xi=(\xi_1, \ldots, \xi_{2n})$ is defined as follows:
\begin{itemize}
\item $s_i=\mathsf{u}$ then $\xi_i=1$;
\item $s_i=\mathsf{d}$, then
\begin{itemize}
\item if $\sigma(i)<i$ (i.e., $i$ is a cycle double descent or cycle peak),
then  $h_{2i-1}>0$; let $\xi_i=m$ if $\sigma(2i)$ is the 
$m$-th isolated vertex in the bottom row of 
$\mathcal G_{2i-2}$ from right-to-left ($1\leq m\leq  \lceil {h_{2i}/2}\rceil$); clearly the value  $i$ will 
contribute $m-1$ crossings  $l<k<i<j$ such that
$l=\sigma(i),\;k=\sigma(j)$;
\item if  $\sigma^{(-1)}(i)\leq i$ (i.e.,
$i$ is a cycle double ascent, cycle peak or fixed point), then 
 $h_{2i}>0$; let $\xi_i=m$ if $\sigma^{(-1)}(i)$ is the $m$-th
 isolated vertex in the top row of $\mathcal G_{2i-2}$ from right-to-left, so 
 $1\leq m\leq \lceil {h_{2i}/2}\rceil$; 
 clearly  the value $i$ will 
contribute $m-1$ crossings  $l<k<i<j$ such that
$l=\sigma^{-1)}(i),\;k=\sigma(j)$, and  $i$
is a record if and only if $m=\lceil {h_{2i}/2}\rceil$.
\end{itemize} 
\end{itemize}
Let $\Phi(\sigma)=(s,\xi)$. Then 
\begin{align*}
\wex(\sigma)&=|\{i\in [n]: s_{2i}=\mathsf{d}\}|, \\
\rec(\sigma)&=|\{i\in [n]: s_{2i}=\mathsf{d}, \, \xi_{2i}=\lceil{h_{2i}/2}\rceil\}|,\\
\cros(\sigma)&=\sum_{i: s_i=\mathsf{d}} (\xi_i-1).
\end{align*}
Therefore, 

\begin{align}
\sum_{\sigma\in
\S_n}\beta^{\rec(\sigma)}y^{\wex(\sigma)}q^{\cros(\sigma)}
&=\sum_{(s,\xi)\in \mathcal{LH}_n} 
\prod_{i: s_i=\mathsf{d}}q^{\xi_i-1}
\prod_{i: s_{2i}=\mathsf{d}}y \beta^{\chi(\xi_{2i}=\lceil{h_{2i}/2}\rceil)}\nonumber\\
&=\sum_{s\in \mathsf{Dyck}_n }\prod_{i: s_i=\mathsf{d}} w(s_i),\label{gf-S}
\end{align}
where $\mathsf{Dyck}_n$ denotes the set of Dyck paths of semilength $n$, and 
the weight of each down step $s_i=\mathsf{d}$ is defined  by
$$
w(s_i)=\begin{cases}
1+q+\cdots +q^{k-1},& \textrm{if  }  h_i=2k,\\ 
y(1+q+\cdots +q^{k-1}+\beta q^k), & \textrm{if } \; h_i=2k+1.
\end{cases}
$$
A folklore theorem~\cite{Fla80} implies that the generating function 
of \eqref{gf-S}  has 
the continued fraction expansion \eqref{eq:co-stieltjes}, and 
 we are done.
\end{proof}

\begin{example}
If $\sigma=4\, 1\,2\, 7\,9\, 6\,5\,8\,3\in \S_9$,  then the Laguerre history $\Phi(\sigma)=(s,\xi)$ is given by
$$
\begin{pmatrix}
s\\
\xi
\end{pmatrix}
=
\begin{pmatrix} 
\mathsf{u\,u}&\mathsf{d\,u}&\mathsf{d\,u}&\mathsf{u\,d}&\mathsf{u\,u}&\mathsf{u\,d}&\mathsf{d\,d}&\mathsf{u\,d}&\mathsf{d\,d}\\
1\,1        &1 \,1      &1        \,1       &1  \,2       &1        \,1      &1       \,1        &1      \,2       &1       \,1       &1\,1 
\end{pmatrix}.
$$
\end{example}



\begin{thm}\label{lin3}  
Let $\alpha\in \N_0$. 
 For nonnegative integers  $n_{1},\ldots, n_{k}$, 
the linearization coefficient 
\begin{align}\label{eq:lin3}
{{\mathcal L}}_q\left(\prod_{k=1}^mL^{(\alpha)}_{n_k}(x;y\,|\,q)\right)\quad \end{align}
is a polynomial in $\N[y, q]$.
\end{thm}
\begin{proof}  In view of the orthogonality~\eqref{eq:orthog}, it suffices to prove the $m=3$ case. 
Indeed, we can derive  the following explicit formula from 
 \cite[Theorem~1]{KSZ11}:
 \begin{multline}\label{linerization3}
{\mathcal L}_q(L^{(\alpha)}_{n_1}(x;y\,|\,q)L^{(\alpha)}_{n_2}(x;y\,|\,q)L^{(\alpha)}_{n_3}(x;y\,|\,q))\\
=n_1!_q\,
n_2!_q\,n_3!_q\sum_{s\geq \max(n_1, n_2, n_3)}
y^s\,\begin{bmatrix}s\\ n_1+n_2+n_3-2s,s-n_3, s-n_2, s-n_1\end{bmatrix}_q\\
\times \begin{bmatrix}\alpha+s\\ s\end{bmatrix}_q\,\sum_{k\geq 0}\begin{bmatrix}n_1+n_2+n_3-2s\\ k\end{bmatrix}_q
y^{k}q^{\binom{k+1} {2}+ \binom{n_1+n_2+n_3-2s-k}{ 2}+k\alpha},
\end{multline}
where the $q$-multinomial coefficients
$$
\begin{bmatrix}a+b+c+d\\ a,b,c,d\end{bmatrix}_q=\frac{(a+b+c+d)!_q}{a!_q\,b!_q\,c!_q\,d!_q}
$$
are known to be  polynomials in $\N[q]$ for integral $a,b,c,d\geq 0$, see
\cite{GJ83}.  Hence,
the right-hand side of \eqref{linerization3} is   a polynomial in $\N[y,q]$, and we are done. 
\end{proof}

For arbitrary $\alpha$, a combinatorial interpretation of \eqref{eq:lin3} was given 
by Foata and Zeilberger~\cite{FZ88} with $y=q=1$, and generalized 
 by the second author~\cite{Zeng92} to  $q=1$ (see also~\cite{Zeng16}), while for   
$\alpha=0$ a combinatorial interpretation of \eqref{eq:lin3} was given 
by Kasraoui et al.~\cite{KSZ11}.
 Thus,  the following problem suggests itself.

 \medskip
\noindent \textbf{Problem.}
\emph{What is the
combinatorial interpretation of \eqref{eq:lin3} for $\alpha\in \N_0$ 
unifying the two  special  cases with $\alpha=0$ or $q=1$?}

\section{Connection with rook polynomials and matching polynomials}
In this section we show how
the model of $\alpha$-Laguerre configurations is connected with the models 
 of non-attacking rook placements and matchings of complete bipartite graphs.
\subsection{Interpretation in rook polynomials}
An $m$ by $n$ board $B$ is a subset of an $m\times n$ grid of cells (or squares). 
A rook is a chessboard piece which takes on rows and columns.  If $r_k$ is the number of 
ways of putting $k$ non-attacking rooks on this board, then the ordinary 
rook polynomial is defined by 
$$
R_{m,n}(x)=\sum_{k}r_kx^k.
$$
Thus, the Laguerre polynomials \eqref{eq:explag} can be written as
\begin{equation}
L^{(\alpha)}_n(x)=(-1)^n n!\, R_{n,n+\alpha}(-x^{-1}).
\end{equation}
   A $k$-\emph{rook placement} on a board $B$ is a subset $C\subset B$ of $k$ cells such that no two cells are in the same row or column of $B$.   We refer the reader to Riordan's  classical book \cite[Chapters~7 and~8]{RI58} for many problems  formulated in terms of configurations of non-attacking rooks on ``chessboards'' of various shapes.

  We label the rows of the grid from  top to bottom and the  columns from left  to right
in the same way as referring to the entries of  an $m\times n$ matrix.  Recall that an \emph{integer partition} is a sequence of positive integers 
$\mu:=(\mu_1,\mu_2,\ldots,\mu_l)$ such that 
$\mu_1\geq\mu_2\geq\cdots\geq\mu_l>0$. We also use the notation 
$\mu=(n_1^{m_1}, \ldots, n_k^{m_k})$ to denote the partition with $m_i$ parts equal to $n_i$ for $i=1, \ldots, k$.
For convenience, we shall identify $\mu$ with its Ferrers board $B_\mu$,
which is defined   as  the  subset $\{(i,j): 1\leq i\leq \mu_j,\; 1\leq j\leq l \}$ of $\N\times \N$.
For a placement $C$ of rooks on $B_\mu$,
 the inversion number  $\inv(C)$ is defined as follows:
 for each rook (cell) in $C$ cross out  all the cells which  are below or to the right of the rook; then $\inv(C)$ is the number of squares of $F_\mu$ that are not crossed out.  An example  is shown in Figure~\ref{fig3}. 
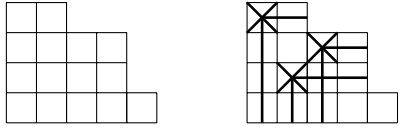
\begin{figure}[t]
\begin{center}
\begin{tikzpicture}[scale=0.4]
\draw[step=1] (0,0) grid (2,4);
\draw[step=1] (2,0) grid (4,3);
\draw[step=1] (4,0) grid (5,1);
\draw[step=1] (8,0) grid (10,4);
\draw[step=1] (10,0) grid (12,3);
\draw[step=1] (12,0) grid (13,1);
\draw (8.5,3.5) node {\Huge$\mathbf{\times}$};
\draw (10.5,2.5) node {\Huge$\mathbf{\times}$};
 \fill (9.5,1.5) node {\Huge$\mathbf{\times}$};
\draw[line width=1] (8.5,3.5)--(8.5,0);
\draw[line width=1] (8.5,3.5)--(10,3.5);
\draw[line width=1] (9.5,1.5)--(9.5,0);
\draw[line width=1] (10.5,2.5)--(12,2.5);
\draw[line width=1] (10.5,2.5)--(10.5,0);
\draw[line width=1] (9.5,1.5)--(12,1.5);
\end{tikzpicture}
\end{center}
\caption{The Ferrers board of shape $\mu=(4,4,3,3,1)$ and a placement~$C$ of three non-attacking rooks  with $\inv(C)=3$.}
\label{fig3}
\end{figure}
\begin{defi} 
For integers $n, k\geq 0$ and $\alpha\geq -1$, let $\mathbf{m}=(m_0, \ldots, m_\alpha)$ and $\mathbf{n}=(n_1, \ldots, n_k)$ be nonnegative integer sequences such that  $m_0+m_1+ \cdots +m_\alpha+ n_1+ \cdots+n_k=n$ with 
$m_i\geq 0$ and $n_j\geq 1$.
We define $\mathcal{B}^{(\alpha)}_{n,k}(\mathbf{m}; \mathbf{n})$ 
as the set of $n\times n$ squares of color shape $B:=(B^{(1)}; B^{(2)})$ with
\begin{subequations}
\label{eq36}
\begin{align}
B^{(1)}:=&(n^{m_0}, \ldots,n^{m_\alpha}),\\
B^{(2)}:=&(n^{n_1}, \ldots, n^{n_k}).
\end{align}
\end{subequations}
By convention,
 if $\alpha=-1$ {\em(}respectively $k=0${\em)},    then $B^{(1)}=\emptyset$ 
{\em(}respectively $B^{(2)}=\emptyset${\em)}.
 Let
\[
\cw(B)=\sum_{i=0}^{\alpha}m_i \quad\text{and}\quad
\cs(B)=\sum_{i=0}^{\alpha}i\cdot m_i.
\]
 Let $\mathcal{B}\mathcal{C}_{n,k}^{(\alpha)}(\mathbf{m}; \mathbf{n})$ denote the set 
 of all ordered pairs $\mathcal{R}=(B, C)$, where 
 $B\in \mathcal{B}^{(\alpha)}_{n,k}(\mathbf{m}; \mathbf{n})$ and $C$ is an 
 $n$-rook placement on  $B$ 
such that 
\begin{align}
\min(C\cap B_1^{(2)})<\min(C\cap B_2^{(2)})<\cdots <\min(C\cap B_k^{(2)}),
\end{align}
where $\min(C\cap B_1^{(2)})$ is the minimum row index of cells in $C\cap B_1^{(2)}$.
 For each block $B_i^{(2)}=(n^{n_i})$, we define 
 $\ind(C\cap B_i^{(2)})$ as the number of rooks in $C\cap B_i^{(2)}$ 
 whose column indices are greater than the column index of the 
 rook which  has the maximum row index in $B_i^{(2)}$, and 
 let $\ind(\mathcal{R})=\sum_{i=1}^{k}\ind(C\cap B_i^{(2)})$.
 Let
 $$
 \mathcal{B}\mathcal{C}_{n,k}^{(\alpha)}=
 \bigcup_{\mathbf{m}, \mathbf{n}}\mathcal{B}\mathcal{C}_{n,k}^{(\alpha)}(\mathbf{m}; \mathbf{n}) \hspace{1cm}\textrm{with}\qquad \sum_{i=0}^\alpha m_i
 +\sum_{j=1}^kn_j=n.
 $$
An element  $\mathcal{R}=(B, C)\in \mathcal{B}\mathcal{C}_{n,h}^{(\alpha)}$ 
is called a colored rook configuration.
\end{defi}
\begin{remark}
One  can imagine that each  column of a board  in $ \mathcal{B}\mathcal{C}_{n,k}^{(\alpha)}(\mathbf{m}; \mathbf{n}) $  is colored with colors in  
$\{0, 1, \ldots, \alpha+k\}$  from left to right as follows:
the first $m_0$ columns get color $0$, the next $m_1$ columns get color $1$, \ldots, the last $n_k$ columns get color $\alpha+k$.
\end{remark}
\begin{thm}\label{th:rook}
The coefficient $\ell^{(\alpha)}_{n,k}(y;q)$ in \eqref{lag-combinatrorial}
 is the following generating polynomial of colored rook configurations   in $\BC_{n,k}^{(\alpha)}$:
$$
\ell^{(\alpha)}_{n,k}(y;q)=\sum_{\mathcal{R}=(B,C)\in\BC_{n,k}^{(\alpha)}}
y^{\cw(B)+\ind(\mathcal{R})}q^{\inv(C)+\cs(B)-\ind(\mathcal{R})}.
$$
\end{thm}
\begin{proof} Let $\lc_{n,k}^{(\alpha)}(\mathbf{m}; \mathbf{n})$ be the set of
$\rho:=(\sigma_0,\dots,\sigma_{\alpha};\, \la_1,\dots, \la_k)\in \lc_{n,k}^{(\alpha)}$ such that $|\rho|=(|\sigma_0|,\ldots,|\sigma_{\alpha}|;\, |\la_1|,\ldots, |\la_k|)=(\mathbf{m}; \mathbf{n})$.
We define the map
$\phi:\lc_{n,k}^{(\alpha)}(\mathbf{m}; \mathbf{n})\longrightarrow 
\BC_{n,k}^{(\alpha)}(\mathbf{m}; \mathbf{n})$ by  $\phi(\rho)=(B, C)$
for $\rho=(\sigma_0,\dots,\sigma_{\alpha};\, \la_1,\dots, \la_k)\in \lc_{n,k}^{(\alpha)}(\mathbf{m}; \mathbf{n})$ as follows:
\begin{itemize}
\item[(i)]
The colored board $B=(B^{(1)},B^{(2)})$ is given by 
$$B^{(1)}=(n^{|\sigma_0|}, \ldots, n^{|\sigma_\alpha|})\quad \text{and}\quad 
B^{(2)}=(n^{|\la_1|}, n^{|\la_2|},\ldots,n^{|\la_k|}).
$$
\item[(ii)] 
If $w:=\hat\sigma_0 \hat\sigma_1\cdots \hat\sigma_{\alpha}\la_1 \la_2\cdots \la_k=
w_1\ldots w_n$, which is a permutation of $[n]$, let $C=\{(j, w_j): j\in [n]\}$.
\end{itemize}
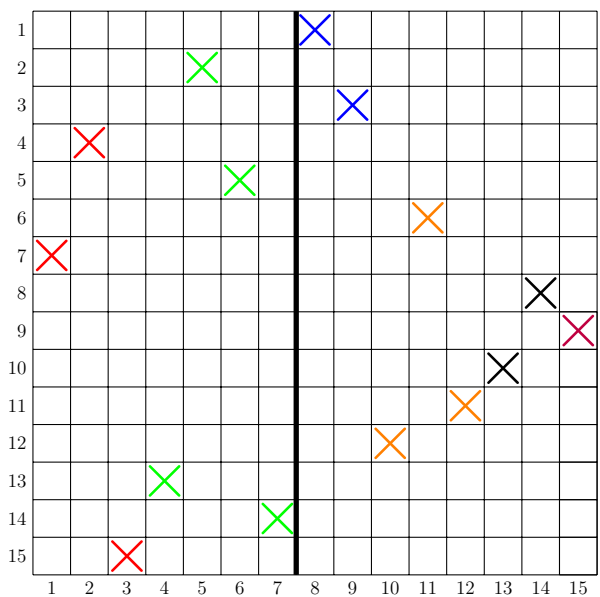
\begin{figure}
\begin{tikzpicture}[scale=0.5]
\draw[step=1] (14,0) grid (29,15);
\draw[step=1] (28,0) grid (29,7);
\draw[line width=2] (21,0)--(21,15);

\node[left,scale=0.6] at (14,0.5){$15$};
\node[left,scale=0.6] at (14,1.5){$14$};
\node[left,scale=0.6] at (14,2.5){$13$};
\node[left,scale=0.6] at (14,3.5){$12$};
\node[left,scale=0.6] at (14,4.5){$11$};
\node[left,scale=0.6] at (14,5.5){$10$};
\node[left,scale=0.6] at (14,6.5){$9$};
\node[left,scale=0.6] at (14,7.5){$8$};
\node[left,scale=0.6] at (14,8.5){$7$};
\node[left,scale=0.6] at (14,9.5){$6$};
\node[left,scale=0.6] at (14,10.5){$5$};
\node[left,scale=0.6] at (14,11.5){$4$};
\node[left,scale=0.6] at (14,12.5){$3$};
\node[left,scale=0.6] at (14,13.5){$2$};
\node[left,scale=0.6] at (14,14.5){$1$};

\node[below,scale=0.6] at (14.5,0){$1$};
\node[below,scale=0.6] at (15.5,0){$2$};
\node[below,scale=0.6] at (16.5,0){$3$};
\node[below,scale=0.6] at (17.5,0){$4$};
\node[below,scale=0.6] at (18.5,0){$5$};
\node[below,scale=0.6] at (19.5,0){$6$};
\node[below,scale=0.6] at (20.5,0){$7$};
\node[below,scale=0.6] at (21.5,0){$8$};
\node[below,scale=0.6] at (22.5,0){$9$};
\node[below,scale=0.6] at (23.5,0){$10$};
\node[below,scale=0.6] at (24.5,0){$11$};
\node[below,scale=0.6] at (25.5,0){$12$};
\node[below,scale=0.6] at (26.5,0){$13$};
\node[below,scale=0.6] at (27.5,0){$14$};
\node[below,scale=0.6] at (28.5,0){$15$};

\draw[red] (14.5,8.5) node {\Huge$\mathbf{\times}$};
\draw[red] (15.5,11.5) node {\Huge$\mathbf{\times}$};
\draw[red]  (16.5,0.5) node {\Huge$\mathbf{\times}$};
\draw[green](17.5,2.5) node {\Huge$\mathbf{\times}$};
\draw[green] (18.5,13.5) node {\Huge$\mathbf{\times}$};
\draw[green] (19.5,10.5) node {\Huge$\mathbf{\times}$};
\draw[green] (20.5,1.5) node {\Huge$\mathbf{\times}$};
\draw[blue] (21.5,14.5) node {\Huge$\mathbf{\times}$};
\draw[blue] (22.5,12.5) node {\Huge$\mathbf{\times}$};
\draw[orange] (23.5,3.5) node {\Huge$\mathbf{\times}$};
\draw[orange] (24.5,9.5) node {\Huge$\mathbf{\times}$};
\draw[orange]  (25.5,4.5) node  {\Huge$\mathbf{\times}$};
\draw (26.5,5.5) node {\Huge$\mathbf{\times}$};
\draw (27.5,7.5) node {\Huge$\mathbf{\times}$};
\draw[purple] (28.5,6.5)  node {\Huge$\mathbf{\times}$};
\end{tikzpicture}
\caption{The colored rook configuration $\mathcal{R}$ corresponding to the $1$-Laguerre 
configuration in Figure~\ref{lagconf2} with  $(\mathbf{m};\mathbf{n}) =((3,4); (2, 3,2,1))$.}
\label{fig2.1}
\end{figure}
It is clear that  $\phi(\rho)\in\BC_{n,k}^{(\alpha)}$, and the procedure is 
reversible. Hence $\phi$ is a bijection.
It is easy to verify that 
$\inv(C)=\inv(\rho)$, $\ind(B^{(2)}_i)=\rl(\la_i)$, and $\cs(B)=\sum_{i=0}^{\alpha}i|\sigma_i|$, which  implies that 
\begin{align*}
{|\underline{\sigma}|+\rl(\lambda)}&=\cw(B)+\ind(\mathcal{R});\\
{\inv(\underline{\sigma}.\underline{\lambda})-\rl(\lambda)+\inv(\underline{\underline{\sigma}})}&=\inv(C)+\cs(B)-\ind(\mathcal{R}).
\end{align*}
The result then follows from Theorem~\ref{key}.
\end{proof}
\begin{example}\label{exam1}
Let  $\rho=((7\,4)(15),(13\,2\,5)(14);1\,3, 12\,6\,11,10\,8,9)\in\lc_{15,4}^{(1)}$.
Then $\phi$ maps $\rho$ to the placement of 15 rooks on the  board $B=({15^7};15^2,10^3,8^2,7)$ shown in Figure~\ref{fig2.1}. We find
$\cw(B)=1,\;\cs(B)=4;\;
\inv(C)=52$ and  $\ind(\mathcal{R})=3$.
\end{example}
\subsection{Interpretation in matching polynomials}
Recall that a  \emph{matching} of a graph $G$  is a set of edges without 
common vertices.  For any graph $G$ with $n$ vertices,  the \emph{matching polynomial} of $G$ is defined by 
$$
m(G, x)
=\sum_{k=0}^{\lfloor n/2\rfloor}(-1)^k m_k x^{n-2k},
$$
where $m_k$ is the number of  $k$-edge matchings of $G$. 
Let  
 $K_{n, m}$ denote  the  set of complete bipartite graphs
on the two disjoint sets $A=[n]$ and 
$B=\{1', \ldots, m'\}$, that is, there is an edge $(a, b)$ if and only if $a\in A$ and $b\in B$.
From the explicit formula \eqref{eq:explag} it is quite easy  to 
derive the connection formula 
\begin{align}\label{lag-matching}
m(K_{n,n+\alpha}, x)= x^{\alpha}L_n^{(\alpha)}(x^{2}),\qquad \alpha\geq -1.
\end{align}
 Godsil and Gutman~\cite{GG81}  proved \eqref{lag-matching}
by showing that the matching polynomials satisfy the same three-term recurrence relation
\eqref{eq:recurr}.  Here we give a simple bijection between our $\alpha$-Laguerre configuration model 
and  the above matching model  of complete bipartite graphs. 
Let $\mathcal{M}^{n-k}_{n,m}$ be the set of matchings of 
$K_{n, m}$ with $n-k$ edges.

\begin{prop} For integers  $n,k\geq 1$ and $\alpha\geq -1$,
there exists an explicit  bijection $
\phi:\: \mathcal{LC}_{n,k}^{(\alpha)}\:\longrightarrow\:
\mathcal{M}^{n-k}_{n,n+\alpha}$.
\end{prop}
\begin{proof}
We construct such a bijection $\phi$.  Let
$\rho=(\sigma_0,\sigma_1,\dots,\sigma_{\alpha};\la_1,\la_2,\dots,\la_k)\in\mathcal{LC}_{n,k}^{(\alpha)}$ be 
an $\alpha$-Laguerre configuration. We define 
a matching
 $\gamma$ of $K_{n, n+\alpha}$ such that
 $(a,b')\in A\times B$ is an edge in $\gamma$  if and only if  
 $(a, b)$ satisfies   one of the  following three conditions:
 \begin{enumerate}
\item $\sigma_0(a)=b$, i.e., the image of $a$ is $b$ through the action of permutation $\sigma_0$;
\item  $a$ and $b$ are consecutive letters in the word
 $\hat\sigma_1(n+1)\hat\sigma_2(n+2)\ldots \hat\sigma_\alpha (n+\alpha)$;
\item $a$ and $b$ are consecutive letters in the word $\la_j$ for some $j\in [k]$.
\end{enumerate}
By convention,  if $\alpha=-1$ (respectively $\alpha=0$) there are no words of types (1) and (2) (respectively type (2)). Since $\sigma_0\sigma_1\ldots \sigma_\alpha \la_1\ldots \la_k$ is a permutation of $[n]$,
it is clear that there are $n-k$ such edges $(a,b')$. The above procedure is obviously reversible.
\end{proof}
\begin{example}
 For the  $1$-Laguerre configuration
 $$\rho=((7\,4)(15),(13\,2\,5)(14);1\,3, 12\,6\,11,10\,8,9)\in\lc_{15,4}^{(1)}$$
 in  Example~\ref{example1}, the corresponding  matching $\gamma$ of $K^{11}_{15, 16}$ is shown in Figure~\ref{matching}.

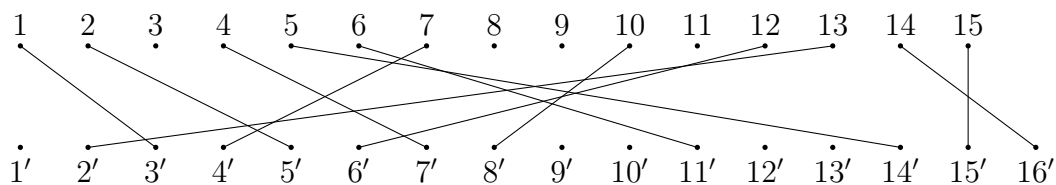
\begin{figure}[t]
\begin{tikzpicture}[scale=0.9]
\node[above] at (1,0.5) {$1$};
\node[above] at (2,0.5) {$2$};
\node[above] at (3,0.5) {$3$};
\node[above] at (4,0.5) {$4$};
\node[above] at (5,0.5) {$5$};
\node[above] at (6,0.5) {$6$};
\node[above] at (7,0.5) {$7$};
\node[above] at (8,0.5) {$8$};
\node[above] at (9,0.5) {$9$};
\node[above] at (10,0.5) {$10$};
\node[above] at (11,0.5) {$11$};
\node[above] at (12,0.5) {$12$};
\node[above] at (13,0.5) {$13$};
\node[above] at (14,0.5) {$14$};
\node[above] at (15,0.5) {$15$};
\fill (1,0.5) circle (1.1pt);
\fill (2,0.5) circle (1.1pt);
\fill (3,0.5) circle (1.1pt);
\fill (4,0.5) circle (1.1pt);
\fill (5,0.5) circle (1.1pt);
\fill (6,0.5) circle (1.1pt);
\fill (7,0.5) circle (1.1pt);
\fill (8,0.5) circle (1.1pt);
\fill (9,0.5) circle (1.1pt);
\fill (10,0.5) circle (1.1pt);
\fill (11,0.5) circle (1.1pt);
\fill (12,0.5) circle (1.1pt);
\fill (13,0.5) circle (1.1pt);
\fill (14,0.5) circle (1.1pt);
\fill (15,0.5) circle (1.1pt);
\fill (1,-1) circle (1.1pt);
\fill (2,-1) circle (1.1pt);
\fill (3,-1) circle (1.1pt);
\fill (4,-1) circle (1.1pt);
\fill (5,-1) circle (1.1pt);
\fill (6,-1) circle (1.1pt);
\fill (7,-1) circle (1.1pt);
\fill (8,-1) circle (1.1pt);
\fill (9,-1) circle (1.1pt);
\fill (10,-1) circle (1.1pt);
\fill (11,-1) circle (1.1pt);
\fill (12,-1) circle (1.1pt);
\fill (13,-1) circle (1.1pt);
\fill (14,-1) circle (1.1pt);
\fill (15,-1) circle (1.1pt);
\fill (16,-1) circle (1.1pt);
\node[below] at (1,-1) {$1'$};
\node[below] at (2,-1) {$2'$};
\node[below] at (3,-1) {$3'$};
\node[below] at (4,-1) {$4'$};
\node[below] at (5,-1) {$5'$};
\node[below] at (6,-1) {$6'$};
\node[below] at (7,-1) {$7'$};
\node[below] at (8,-1) {$8'$};
\node[below] at (9,-1) {$9'$};
\node[below] at (10,-1) {$10'$};
\node[below] at (11,-1) {$11'$};
\node[below] at (12,-1) {$12'$};
\node[below] at (13,-1) {$13'$};
\node[below] at (14,-1) {$14'$};
\node[below] at (15,-1) {$15'$};
\node[below] at (16,-1) {$16'$};
\draw (1,0.5)--(3,-1);
\draw (2,0.5)--(5,-1);
\draw (4,0.5)--(7,-1);
\draw (5,0.5)--(14,-1);
\draw (6,0.5)--(11,-1);
\draw (7,0.5)--(4,-1);
\draw (10,0.5)--(8,-1);
\draw (12,0.5)--(6,-1);
\draw (13,0.5)--(2,-1);
\draw (15,0.5)--(15,-1);
\draw (14,0.5)--(16,-1);
\end{tikzpicture}
\caption{The matching corresponding to the 
 $1$-Laguerre configuration in Figure~\ref{lagconf2}}
\label{matching}
\end{figure}
\end{example}

\begin{remark}
We leave it to the interested 
reader to find the $(q,y)$-version of the above matching polynomials for 
$(q,y)$-Laguerre polynomials $L^{(\alpha)}_{n}(x;y\,|\,q)$.
\end{remark}
\section*{Acknowledgement} 
We    thank    the    anonymous    reviewers
    for     their careful    reading    of    our    manuscript    and    their    many    
helpful    comments and    suggestions.
%

\end{document}